\documentclass[12pt]{amsart}
\usepackage{amssymb}
\theoremstyle{plain}
 \newtheorem{thm}{Theorem}
 
 \newtheorem{lem}[thm]{Lemma}
 
\theoremstyle{definition}

\theoremstyle{remark}

\allowdisplaybreaks
\newcommand{\NaturalNumber}{\mathbb N}
\newcommand{\RealNumber}{\mathbb R}

\begin{document}
\title[The Brouwer fixed point theorem]
{An easily verifiable proof \\of the Brouwer fixed point theorem}
\author[Y. Takeuchi]{Yukio Takeuchi}
\address[Y. Takeuchi]{Takahashi Institute for Nonlinear Analysis,
 1-11-11, Nakazato, Minami, Yokohama 232-0063, Japan}
\email{aho31415@yahoo.co.jp}
\author[T. Suzuki]{Tomonari Suzuki}
\address[T. Suzuki]{Department of Basic Sciences,
Kyu\-shu Institute of Technology,
To\-bata, Kita\-kyu\-shu 804-8550, Japan}
\email{suzuki-t@mns.kyutech.ac.jp}
\date{}


\begin{abstract}
We give a remarkably elementary proof of the Brouwer fixed point theorem.
The proof is verifiable for most of the mathematicians.
\end{abstract}
\maketitle


\section{Introduction}
\label{SC:introduction}

The following theorem is referred to as the
 {\it Brouwer fixed point theorem}.

\begin{thm}[Brouwer \cite{REF:Brouwer1912_MathAnn}, Hadamard \cite{Hadamard}]
\label{THM:Brouwer}
Let $n \in \NaturalNumber$ and let $g$ be a continuous mapping on $[0,1]^n$.
Then there exists $z \in [0,1]^n$ such that $g(z)=z$.
\end{thm}

Theorem \ref{THM:Brouwer} is used in numerous fields of mathematics.
So we can consider
 Theorem \ref{THM:Brouwer} is one of the most useful theorems in mathematics.
There are many proofs of Theorem \ref{THM:Brouwer}.
For example,
 the proof based on the Sperner lemma \cite{Sperner}
 is very excellent in a geometric sense.
See also Stuckless \cite{REF:StucklessTara2003} and references therein.

On the other hand,
 in \cite[page 4]{REF:StucklessTara2003},
 it is said that
``It has been estimated that
 while 95\% of mathematicians can state Brouwer's theorem,
 less than 10\% know how to prove it.''
This implies that
 Theorem \ref{THM:Brouwer} is very useful, however,
 the proofs we have obtained are not easy.

In this paper, motivated by this fact,
 we give a proof of Theorem \ref{THM:Brouwer}.
Our proof is so easy that
 most of the mathematicians can verify it.
In our proof, we only use the Bolzano-Weierstrass theorem and
 the fact that an odd number plus an even number is odd.
We do not need any geometric intuition.
Therefore the authors believe that
 even undergraduate students are able to understand our proof.


\section{Preliminaries}
\label{SC:preliminaries}

Throughout this paper,
 we denote by $\NaturalNumber$ the set of all positive integers
 and by $\RealNumber$ the set of all real numbers.
We define $N(i,j)$ by
 $$ N(i,j)
 = \{ k : \; k \in \NaturalNumber \cup \{ 0 \}, \; i \leq k \leq j \} . $$
For $x \in \RealNumber^n$, $(x)_i$ denotes the $i$-th coordinate of $x$.
For an arbitrary set $B$,
 we also denote by $\# B$ the cardinal number of $B$.

Let $L$ be an arbitrary set, $n \in \NaturalNumber$ and $k \in N(0,n)$.
Let $\ell$ be a mapping from $L$ into $N(0,n)$.
We call such a mapping a {\it labeling}.
Then a subset $B$ of $L$ is called {\it $k$-fully labeled} if
 $\# B = k+1$ and $\ell(B) = N(0,k)$.
The following lemma is very fundamental, however, it plays an important role.

\begin{lem}
\label{LEM:fully_labeled}
Let $L$ be an arbitrary set, $n \in \NaturalNumber$ and $k \in N(1,n)$.
Let $\ell$ be a labeling from $L$ into $N(0,n)$.
Let $B$ be a subset of $L$ with $\# B = k+1$ and $\ell(B)\subset N(0,k)$.
Then the following hold{\rm :}
\begin{enumerate}
\renewcommand{\labelenumi}{(\roman{enumi})}
\item
$B$ includes at most two subsets which are $(k-1)$-fully labeled.
\item
$B$ includes exactly one subset which is $(k-1)$-fully labeled
 if and only if
 $B$ is $k$-fully labeled.
\end{enumerate}
\end{lem}

\begin{proof}
Suppose there exists a subset $C$ of $B$ which is $(k-1)$-fully labeled.
Since $\# C = k$, we can suppose $B = C \sqcup \{ b \}$,
 where $\sqcup$ represents `disjoint union'.
In the case where $\ell(b) < k$,
 only two subsets
 $C$ and $\big( C \setminus \ell^{-1}( \ell(b) ) \big) \cup \{ b \}$
 of $B$ are $(k-1)$-fully labeled.
In the other case, where $\ell(b) = k$,
 $B$ is $k$-fully labeled and
 $C$ is the only subset of $B$ which is $(k-1)$-fully labeled.
We have shown (i).
{}From the above observation, (ii) is obvious.
\end{proof}


\section{A Labeling Theorem}
\label{SC:labeling-theorem}

In this section, we fix $n, m \in \NaturalNumber$.
We set $n$ points $e_1, \cdots, e_n \in \RealNumber^n$ by
$$ e_1=(1/m,0,\cdots ,0), \quad
 e_2=(0,1/m,0,\cdots ,0), \quad
 \cdots ,\quad
 e_n=(0,\cdots ,0,1/m). $$
Define $n+1$ subsets $L_0, \cdots, L_n$ of $[0, 1]^n$ by
 $L_0 = \{ 0 \}$ and
\begin{equation}
\label{EQU:L}
 L_k
 = \big\{ \textstyle\sum_{i=1}^k \alpha_i e_i : \;
 \alpha _i \in N(0,m) \big\}
\end{equation}
 for $k \in N(1,n)$.
A labeling $\ell$ from $L_n$ into $N(0,n)$ is called {\it Brouwer} if
 the following two conditions are satisfied:
\begin{enumerate}
\renewcommand{\labelenumi}{(B\arabic{enumi})}
\item
If $(x)_k = 0$ for some $k \in N(1,n)$, then $\ell(x) \neq k$.
\item
If $(x)_k = 1$ for some $k \in N(1,n)$, then $\ell(x) \geq k$.
\end{enumerate}
It is obvious that $\ell(x) \leq k$ for $k \in N(0,n)$ and $x \in L_k$.
A subset $B$ of $L_k$ is said to be a {\it $k$-string} if
 there exist $x_0, \cdots, x_k \in L_k$ and a bijection $\sigma$ on $N(1,k)$
 such that
 $B = \{ x_0, \cdots, x_k\}$ and
 $ x_j = x_0 + \sum_{i=1}^j e_{\sigma(i)} $ for $j \in N(1,k)$.
The following are obvious:
\begin{align}
&\textstyle
 x_k = x_0 + \sum_{i=1}^k e_i ,
\label{EQU:string-1} \\*
&\textstyle
 \sum_{i=1}^n (x_j)_i = \sum_{i=1}^n (x_0)_i + j/m
 \quad \text{for } j \in N(1,k) .
\label{EQU:string-2}
\end{align}
We sometimes write $B = \langle x_0, \cdots, x_k \rangle$.
We note that $B$ is a $0$-string if and only if $B = \{ 0 \}$.

\begin{thm}
\label{THM:label}
Let $\ell$ be a Brouwer labeling from $L_n$ into $N(0,n)$.
Then there exists an $n$-string which is $n$-fully labeled.
\end{thm}

We note that Theorem \ref{THM:label} is connected with
 the Sperner lemma \cite{Sperner}.
Before proving Theorem \ref{THM:label}, we need some preliminaries.

\begin{lem}
\label{LEM:induction-0}
There exists exactly one $0$-string which is $0$-fully labeled.
\end{lem}

\begin{proof}
Since $\ell(0) = 0$,
 $\langle 0 \rangle$ is a $0$-string which is $0$-fully labeled.
\end{proof}

\begin{lem}
\label{LEM:string-k-1}
Let $C$ be a $(k-1)$-string for some $k \in N(1,n)$.
Then there exists exactly one $k$-string which includes $C$.
\end{lem}

\begin{proof}
Suppose $C = \langle x_0, \cdots, x_{k-1} \rangle$.
It follows from \eqref{EQU:string-1} and $(x_i)_k = 0$ for $i \in N(0,k-1)$
 that only
 $$ B = \langle x_0, \cdots, x_{k-1}, x_{k-1} + e_k \rangle $$
 is a $k$-string which includes $C$.
\end{proof}

\begin{lem}
\label{LEM:string-k}
Let $B$ be a $k$-string for some $k \in N(1,n)$ and
 let $C$ be a subset of $B$ which is $(k-1)$-fully labeled.
Then the following hold{\rm :}
\begin{enumerate}
\renewcommand{\labelenumi}{(\roman{enumi})}
\item
If $C \subset L_{k-1}$,
 then $C$ is a $(k-1)$-string and
 there exists exactly one $k$-string which includes $C$.
\item
If $C \not\subset L_{k-1}$,
 then there exist exactly two $k$-strings which include $C$.
\end{enumerate}
\end{lem}

\begin{proof}
Suppose $B = \langle x_0, \cdots, x_k \rangle$,
 $ x_j = x_0 + \sum_{i=1}^j e_{\sigma(i)} $ and
 $B = C \sqcup \{ x_h \}$.
In the case of (i),
 it is obvious that $h = k$ and $\sigma(k) = k$.
So $C$ is a $(k-1)$-string.
By Lemma \ref{LEM:string-k-1},
 only $B$ is a $k$-string which includes $C$.
In the case of (ii), we consider three cases:
\begin{itemize}
\item $h = 0$
\item $0 < h < k$
\item $h = k$
\end{itemize}
In the first case, arguing by contradiction,
 we assume $(x_k)_{\sigma(1)} = 1$.
Then $(x_1)_{\sigma(1)} = \cdots = (x_k)_{\sigma(1)} = 1$,
 which imply $\ell(x_1) \geq \sigma(1) > 0, \cdots,
 \ell(x_k) \geq \sigma(1) > 0$
 by (B2).
Thus $C$ is not $(k-1)$-fully labeled.
This is a contradiction.
Thus $(x_k)_{\sigma(1)} < 1$.
Therefore
 $$ B' = \langle x_1, \cdots, x_k, x_k + e_{\sigma(1)} \rangle $$
 is another $k$-string and includes $C$.
In the second case,
 $$ B' = \langle x_0, \cdots, x_{h-1},
 x_{h-1}+e_{\sigma(h+1)}, x_{h+1}, \cdots, x_k \rangle $$
 is another $k$-string and includes $C$.
In the third case,
 arguing by contradiction,
 we assume $(x_0)_{\sigma(k)} = 0$.
Then $(x_0)_{\sigma(k)} = \cdots = (x_{k-1})_{\sigma(k)} = 0$,
 which imply $\ell(x_0) \neq \sigma(k), \cdots, \ell(x_{k-1}) \neq \sigma(k)$
 by (B1).
Since $\sigma(k) = k$ implies $C \subset L_{k-1}$, we have $\sigma(k) < k$.
So $C$ is not $(k-1)$-fully labeled.
This is a contradiction.
Thus $(x_0)_{\sigma(k)} > 0$.
Therefore
 $$ B' = \langle x_0 - e_{\sigma(k)}, x_0, \cdots, x_{k-1} \rangle $$
 is another $k$-string and includes $C$.
By \eqref{EQU:string-1} and \eqref{EQU:string-2},
 it is impossible that three $k$-strings include $C$.
\end{proof}

By Lemmas \ref{LEM:string-k-1} and \ref{LEM:string-k},
 we obtain the following.

\begin{lem}
\label{LEM:fully_labeled-2}
Let $C$ be a subset of $L_k$ which is $(k-1)$-fully labeled
 for some $k \in N(1,n)$.
Then the following hold{\rm :}
\begin{enumerate}
\renewcommand{\labelenumi}{(\roman{enumi})}
\item
$C$ is included by at most two $k$-strings.
\item
$C$ is included by exactly one $k$-string
 if and only if
 $C$ is a $(k-1)$-string.
\end{enumerate}
\end{lem}

\begin{lem}
\label{LEM:induction-k}
Let $k \in N(1,n)$.
Assume
 there exist exactly odd $(k-1)$-strings which are $(k-1)$-fully labeled.
Then there exist exactly odd $k$-strings which are $k$-fully labeled.
\end{lem}

\begin{proof}
Define four sets $S_1$, $S_2$, $T_1$ and $T_2$ as follows:
$B \in S_1$ if and only if
 $B$ is a $k$-string which includes exactly one $(k-1)$-fully labeled subset.
$B \in S_2$ if and only if
 $B$ is a $k$-string which includes exactly two $(k-1)$-fully labeled subsets.
$C \in T_1$ if and only if
 $C$ is a $(k-1)$-fully labeled subset included by exactly one $k$-string.
$C \in T_2$ if and only if
 $C$ is a $(k-1)$-fully labeled subset included by exactly two $k$-strings.
By Lemma \ref{LEM:fully_labeled} (ii),
 we note that $B \in S_1$ if and only if
 $B$ is a $k$-string which is $k$-fully labeled.
By Lemma \ref{LEM:fully_labeled-2} (ii),
 we also note that $C \in T_1$ if and only if
 $C$ is a $(k-1)$-string which is $(k-1)$-fully labeled.
With double-counting,
 we count the number of $(k-1)$-fully labeled subsets in $k$-strings.
Then by Lemmas \ref{LEM:fully_labeled} (i) and \ref{LEM:fully_labeled-2} (i),
 we have
 $$ \# S_1 + 2 \#S_2 = \# T_1 + 2 \# T_2 . $$
Since $\# T_1$ is odd, we obtain $\# S_1$ is odd.
\end{proof}

\begin{proof}[Proof of Theorem \ref{THM:label}]
By Lemmas \ref{LEM:induction-0} and \ref{LEM:induction-k},
There exist exactly odd $n$-strings which are $n$-fully labeled.
Since $0$ is not odd,
 there exists at least one $n$-string which is $n$-fully labeled.
\end{proof}


\section{A Proof}
\label{SC:proof}

We give a proof of Theorem \ref{THM:Brouwer}.

\begin{proof}[Proof of Theorem \ref{THM:Brouwer}]
Define $n$ functions $g_1, \cdots, g_n$ from $[0, 1]^n$ into $[0, 1]$ by
 $ g_k(x) = \big( g(x) \big)_k $ for $k \in N(1,n)$.
It is obvious that $g_k$ is continuous.

We fix $m \in \NaturalNumber$ and define $L_n$ by \eqref{EQU:L}.
Define a labeling $\ell$ from $L_n$ into $N(0,n)$ by
 $$ \ell(x) = \max \{ k \in N(1,n) : \; (x)_k > 0, \; g_k(x) \leq (x)_k \} , $$
 where $\max \varnothing = 0$.
{}From the definition of $\ell$, (B1) obviously holds.
If $(x)_k = 1$, then $g_k(x) \leq (x)_k$, which implies $\ell(x) \geq k$.
So (B2) holds and hence $\ell$ is Brouwer.
We note that if $\ell(x) = 0$, then $(x)_j \leq g_j(x)$ for $j \in N(1,n)$.
By Theorem \ref{THM:label},
 there exists an $n$-string $B^{(m)}$ which is $n$-fully labeled.
Let $y_0^{(m)}, \cdots, y_n^{(m)} \in [0, 1]^n$ satisfy
 $B^{(m)} = \{ y_0^{(m)}, \cdots, y_n^{(m)} \}$
 and $\ell(y_k^{(m)}) = k$ for $k \in N(0,n)$.

Since $[0,1]^n$ is compact, by the Bolzano-Weierstrass theorem,
 $\{ y_0^{(m)} \}$ has a convergent subsequence.
Without loss of generality,
 we may assume $\{ y_0^{(m)} \}$ itself converges to some $z \in [0, 1]^n$.
Since the diameter of $B^{(m)}$ is $\sqrt{n \;} / m$,
 $\{ y_k^{(m)} \}$ also converges to $z$ for $k \in N(1,n)$.
Since
 $(y_0^{(m)})_j \leq g_j(y_0^{(m)})$ for $j \in N(1,n)$ and
 $g_k(y_k^{(m)}) \leq (y_k^{(m)})_k$ for $k \in N(1,n)$,
 we have $g_k(z) = (z)_k$ for all $k \in N(1,n)$.
Thus $g(z)=z$.
\end{proof}


\end{document}